\documentclass[twoside,11pt]{article}

\usepackage{graphicx, subfigure}
\usepackage[top=1in,bottom=1in,left=1in,right=1in]{geometry}
\usepackage[sort&compress,numbers]{natbib} 
\usepackage{amsfonts, amsmath, amssymb, amsthm, constants, bbm}
\usepackage{MnSymbol}
\usepackage{enumitem}
\usepackage[mathscr]{euscript}
\usepackage{pifont}

\usepackage{color}
\definecolor{darkred}{RGB}{100,0,0}
\definecolor{darkgreen}{RGB}{0,100,0}
\definecolor{darkblue}{RGB}{0,0,150}

\usepackage{hyperref}
\hypersetup{colorlinks=true, linkcolor=darkred, citecolor=darkgreen, urlcolor=darkblue}
\usepackage{url}

\def\Beta{\mathsf{B}}
\def\hyp{\mathscr{H}}

\newtheorem{thm}{Theorem}
\newtheorem{prp}{Proposition}
\newtheorem{lem}{Lemma}
\newtheorem{cor}{Corollary}

\theoremstyle{remark}

\newtheorem{rem}{Remark}

\def\beq{\begin{equation}} % \setcounter{equation}{1}}
\def\eeq{\end{equation}}
\def\beqn{\begin{eqnarray*}}
\def\eeqn{\end{eqnarray*}}
\def\Bitem{\begin{itemize}\setlength{\itemsep}{.2in}}
\def\bitem{\begin{itemize}\setlength{\itemsep}{.05in}}
\def\eitem{\end{itemize}}
\def\Benum{\begin{enumerate}\setlength{\itemsep}{.2in}}
\def\benum{\begin{enumerate}\setlength{\itemsep}{.05in}}
\def\eenum{\end{enumerate}}
\def\bmult{\begin{multline*}}
\def\emult{\end{multline*}}
\def\bcenter{\begin{center}}
\def\ecenter{\end{center}}
\def\bframe{\begin{frame}}
\def\eframe{\end{frame}}

\newcommand{\thmref}[1]{Theorem~\ref{thm:#1}}
\newcommand{\prpref}[1]{Proposition~\ref{prp:#1}}

\newcommand{\lemref}[1]{Lemma~\ref{lem:#1}}
\newcommand{\secref}[1]{Section~\ref{sec:#1}}
\newcommand{\figref}[1]{Figure~\ref{fig:#1}}

\newcommand{\remref}[1]{Remark~\ref{rem:#1}}

\DeclareMathOperator*{\argmax}{arg\, max}
\DeclareMathOperator*{\argmin}{arg\, min}
\def\cK{\mathcal{K}}

\def\cT{\mathcal{T}}

\def\bbP{\mathbb{P}}

\def\bbR{\mathbb{R}}

\renewcommand{\P}{\operatorname{\mathbb{P}}}

\def\iid{\stackrel{\rm iid}{\sim}}
\def\Bin{\text{Bin}}

\def\eps{\varepsilon}

\def\iff{\ \Leftrightarrow \ }

\def\1{\mathbbm{1}}

\definecolor{purple}{rgb}{0.4,.1,.9}

\newcommand\blfootnote[1]{%
  \begingroup
  \renewcommand\thefootnote{}\footnote{#1}%
  \addtocounter{footnote}{-1}%
  \endgroup
}

\begin{document}
\thispagestyle{empty}

\title{Detection of Sparse Mixtures: \\ Higher Criticism and Scan Statistic}
\author{Ery Arias-Castro \and Andrew Ying}
\date{}
\maketitle

\blfootnote{Both authors are with the Department of Mathematics, University of California, San Diego, USA.  Contact information is available \href{http://math.ucsd.edu/\~eariasca}{here} and \href{http://www.math.ucsd.edu/people/graduate-students/}{here}.  This work was partially supported by a grant from the US Office of Naval Research (N00014-13-1-0257).}

\begin{abstract}
We consider the problem of detecting a sparse mixture as studied by Ingster (1997) and Donoho and Jin (2004).  We consider a wide array of base distributions.  In particular, we study the situation when the base distribution has polynomial tails, a situation that has not received much attention in the literature.  Perhaps surprisingly, we find that in the context of such a power-law distribution, the higher criticism does not achieve the detection boundary.  However, the scan statistic does. 
\end{abstract}

\section{Introduction} \label{sec:intro}
We consider the problem of detecting a sparse mixture.  A simple variant of the problem can be formulated as follows.  Let $F$ be a continuous distribution function on the real line, and $\eps \in (0,1/2]$ and $\mu > 0$.
The problem is to test
\beq\label{h0}
\hyp_0^n: X_1, \dots, X_n \iid F,
\eeq
versus
\beq\label{h1}
\hyp_1^n: X_1, \dots, X_n \iid (1-\eps) F(\cdot) + \eps F(\cdot - \mu).
\eeq

Mixtures models such as in \eqref{h1} have been considered for quite some time, particularly in the context of robust statistics, where they are known as contamination models \citep[Eq 1.22]{huber2009robust}.

Rather, our contribution is in line with the testing problems studied by \citet{ingster1997some} in the context of the normal sequence model, where $F$ above corresponds to the standard normal distribution.
In that setting, Ingster considered the following parameterization
\beq\label{param}
\eps = \eps_n = n^{-\beta}, \qquad
\mu = \mu_n = \sqrt{2 r \log n},
\eeq
for some $\beta > 0$ and $r > 0$.  
The advantage of this parameterization is that, holding $\beta$ and $r$ fixed, the situation admits a relatively simple description.  
Indeed, since both the null and the alternative hypotheses are simple, by the Neyman-Pearson Lemma, the likelihood ratio test (set at level $\alpha$) is most powerful.  Ingster studied the large-sample behavior of this test procedure and discovered that, in the case where $\beta > 1/2$, when $r < \rho(\beta)$, the test is powerless in the sense of achieving power $\alpha$, while when $r > \rho(\beta)$, the test was fully powerful in the sense of achieving power 1, where the function $\rho$ is given by
\beq \label{lower}
\rho(\beta) := \begin{cases} 
\beta - 1/2, & 1/2 < \beta \leq 3/4, \\ 
(1 - \sqrt{1 - \beta})^2, & 3/4 < \beta < 1.
\end{cases}
\eeq
Thus the existence of a detection boundary in the $(\beta, r)$ plane given by $r = \rho(\beta)$.
In such a situation, we will say that a test procedure `achieves the detection boundary', or is `first-order optimal' (or simply `optimal'), if it is fully powerful when $r > \rho(\beta)$.

Such detection boundaries where derived for other models, for example, in \citep{cai2014optimal,cai2011optimal,donoho2004higher}.
We also mention that the situation where $\beta \le 1/2$ is also well-understood, but quite different, and will not be considered here.  Most of the literature has focused on the more interesting setting where $\beta > 1/2$ and we do the same here.  

\subsection{Threshold tests}
After determining what one can hope for, it becomes of interest to understand what one can achieve with less information.  Indeed, the likelihood ratio test requires knowledge of all the quantities and objects defining the testing problem, in this case $(F,\eps,\mu)$, and even in the present stylized setting we might want to know what can be done when some of this information is missing, in particular what defines the alternative, namely $(\eps,\mu)$.  
(The case where $F$ is also unknown has attracted much less attention.  We discuss it in \secref{discussion}.)

When $F$ is known, the problem is that of goodness-of-fit testing, albeit with alternatives of the form \eqref{h1} in mind.
\citet{donoho2004higher} opened this investigation with the analysis of various tests, including the max test based on $\max_i X_i$ and a variant of the Anderson-Darling test \citep{anderson1952asymptotic}.  
Seeing as a problem of multiple testing based on p-values defined as $U_i = 1 - F(X_i)$, the max test coincides with the Tippett-\v{S}id\'ak test combination test, while the Anderson-Darling test coincides with a proposal by Tukey called the higher criticism (HC).
More recently, \citet{moscovich2016exact} analyzed a goodness-of-fit (BJ) test proposed by \citet{berk1979goodness} in the same setting.
For $t \in \bbR$, define 
\beq\label{counts}
N_n(t) = \#\{i : X_i \ge t\}.
\eeq
%Also, let $\bar F = 1 - F$, which is the survival function of $F$.
We note that, under the null hypothesis, $N_n(t)$ is binomial with parameters $(n, 1-F(t))$, which motivates the test that rejects for large values of 
\beq
\sup_{t : F(t) \le 1/2} \frac{N_n(t) - n (1-F(t))}{\sqrt{n F(t) (1-F(t)) + 1}}.
\eeq
This is one of many possible variants of HC.\footnote{\quad 
The constraint `$F(t) \le 1/2$ can be replaced by $F(t) \le \gamma$, where $\gamma$ can be taken to be smaller, say $\gamma = 0.05$.  
The `$+1$' in the denominator is roughly equivalent to adding the constraint $F(t) \ge 1/n$, which \citet{donoho2004higher} recommend for reasons of stability.  In any case, this variant performs as well (to first order) as any other variant of HC considered in the literature, at least in all the regimes commonly considered.}  

Let $U_{(1)} \le \cdots \le U_{(n)}$ denote the ordered $U_i$'s.  
We note that, under the null hypothesis, $U_{(i)}$ has the beta distribution with parameters $(i, n-i+1)$, which motivates the definition of BJ, rejecting for small values of 
\beq\label{BJ}
\min_{i \in [n]} P_i,
\eeq
where $P_i := \Beta(U_{(i)}; i, n-i+1)$ and $\Beta(\cdot; a,b)$ denotes the distribution function of the beta distribution with parameters $(a,b)$.

The verdict is the following.  
In the normal setting, HC and BJ achieve the detection boundary in the full range $\beta > 1/2$, while the max test is only able to achieve the detection in the upper half of the range $\beta > 3/4$.
The same extends to other models, in particular to generalized Gaussian models where $F$ has density proportional to $\exp(-|x|^a/a)$ for some $a > 1$.  (The case $a \le 1$ is qualitatively different.  HC and BJ are still first-order optimal while the max test is suboptimal everywhere.)

These tests are all threshold tests, where we define a threshold test as any test with a rejection region of the form
$\bigcup_{t \in \cT} \{N_n(t) \ge c_t\},$
for some subset $\cT \subset \bbR$ and some critical values  $c_t > 0$.
More broadly, any combination test that we know of that is discussed in the multiple-testing literature is a threshold test.  (This includes the tests proposed by Fisher, Lipt\'ak-Stouffer, Tippett-\v{S}id\'ak, Simes, and more.)
Thus it might be of interest to understand what can be achieved with a threshold test.  In this regard, it is useful to examine how one would optimize such an approach if one had perfect knowledge of the model.  
Let $\phi_t$ denote the test with rejection region $\{N_n(t) \ge c_t\}$, where 
\beq\label{critical}
c_t := \min\big\{c \ge 0: \P_0(N_n(t) \ge c) \le \alpha\big\}.
\eeq
We define the oracle threshold test as the test $\phi_{t*}$, where 
\beq\label{threshold}
t_* := \argmax_{t \in \bbR} \P_1(N_n(t) \ge c_t),
\eeq
with $\bbP_0$ denoting the distribution under the null \eqref{h0} and $\bbP_1$ that under the alternative \eqref{h1}.
(Here and elsewhere, $\alpha$ denotes the desired significance level.)
Note that computing $c_t$ only requires knowledge of $F$, while computing $t_*$ requires knowledge of the entire model, namely $(F,\eps,\mu)$.  
Thus the construction of the test $\phi_{t*}$ relies on the oracle knowledge of $(\eps,\mu)$.

\subsection{Scan tests}
Detection problems arise in a variety of contexts and in very many applications.  An important example is in spatial statistics (itself a rather wide area), where the detection of `hot spots', meaning areas of unusually high concentration, has been considered for quite some time \citep{kulldorff1997spatial}.
An early contribution to this literature is that of \citet{naus1965distribution}, who considered the distribution of the maximum number of points in an interval of given length (say $\ell$) when the points are drawn iid from the uniform distribution on $[0,1]$.  This would nowadays be referred to as the scan statistic and arises when testing the null that the points are uniformly distributed in $[0,1]$ against the (composite) alternative that there is an sub-interval of length $\ell$ with higher intensity.  Settings where sub-interval length is unknown have been considered \citep{arias2005near}.

For $s \le t$, define $N_n[s,t] = \# \{i: X_i \in [s,t]\}$ and $F[s,t] = F(t) - F(s)$. 
We note that, under the null hypothesis, $N_n[s,t]$ is binomial with parameters $(n, F[s,t])$, which motivates the test that rejects for large values of 
\beq\label{HC_scan}
\sup_{(s,t) : F[s,t] \le 1/2} \frac{N_n[s,t] - n F[s,t]}{\sqrt{n F[s,t] (1 -F[s,t]) + 1}}.
\eeq
Although there are many possible variants, this is the one we will be working with.  

We note that, under the null hypothesis, for any pair of indices $i < j$, $U_{(j)} - U_{(i)}$ has the beta distribution with parameters $(j-i, n-j+i+1)$ --- see \citep[Th 11.1]{gibbons2011nonparametric}.  This motivates the definition of the scan test which rejects for small values of 
\beq \label{BJ_scan}
\min_{1 \le i < j \le n} P_{i,j},
\eeq
where $P_{i,j} := \Beta(U_{(j)} - U_{(i)}; j-i, n-j+i+1)$.

In general, we define a scan test as any test with region rejection of a the form $\bigcup_{(s,t) \in \cK} \{N_n[s,t] \ge c_{s,t}\}$, where $\cK$ is a subset of $\{(s,t) : s < t\}$ and $c_{s,t} \ge 0$ are critical values.  
Let $\phi_{s,t}$ denote the test with rejection region $\{N_n[s,t] \ge c_{s,t}\}$, where 
\beq\label{critical-pair}
c_{s,t} := \min\big\{c \ge 0 : \P_0(N_n[s,t] \ge c) \le \alpha\big\}.
\eeq
We define the oracle scan test as the test $\phi_{s_\bullet,t_\bullet}$, where 
\beq\label{scan}
(s_\bullet,t_\bullet) := \argmax_{s < t} \P_1(N_n[s,t] \ge c_{s,t}).
\eeq
Indeed, $\phi_{s_\bullet,t_\bullet}$ relies on oracle knowledge of $(\eps, \mu)$.

To the best of our knowledge, this is the first time that such tests are considered in the line of work that concerns us here with roots in the work of \citet{ingster1997some} and \citet{donoho2004higher}.
The main reason for considering these tests in the present context is that they happen to be first-order optimal, not only in the models considered in the literature (such as generalized Gaussian), but also in power-law models where $F$ has fat tails (e.g., Cauchy or Pareto), whereas  threshold tests fail are suboptimal for such models.  

\subsection{Content}
For simplicity and the sake of clarity, we will focus on oracle-type, rather than likelihood ratio, performance bounds.  The former are indeed more transparent and can be obtained under more generality and with simpler arguments.  Also our main intention here is to compare what can be achieved with threshold tests compared to the more general scan tests, defined next, and comparing the corresponding oracle tests seems more appropriate.  

In \secref{oracle}, we study the oracle threshold test and the oracle scan test.  We then consider a number of models.
In \secref{scan}, we consider the two scan tests described above and compare them to the oracle scan test.
In \secref{numerics}, we present the result of some numerical experiments that illustrate our theory.
We briefly discuss the performance of the likelihood ratio test and that of nonparametric approaches in \secref{discussion}.

\section{Oracle threshold test and oracle scan test} \label{sec:oracle}

In this section we state and prove some basic results for the oracle threshold and oracle scan tests.

\subsection{Power monotonicity}
It is natural to guess that the testing \eqref{h0} versus \eqref{h1} becomes easier as the shift $\mu$ increases.  This is indeed the case, at least from the point of view of both oracle tests.

\begin{prp} \label{prp:threshold_monotonicity}
The oracle threshold test has monotonic power in the shift.
\end{prp}

\begin{proof}
We assume that $\eps > 0$ is fixed and let $\P_\mu$ denote the data distribution under the alternative \eqref{h1}.  Take $\mu_1 \le \mu_2$ and let $t_k$ denote the oracle threshold \eqref{threshold} for $\mu_k$, so that the oracle test for $\mu_k$, meaning $\phi_{t_k}$, has rejection region $\{N_n(t_k) \ge c_{t_k}\}$ and power $\pi_k := \P_{\mu_k}(N_n(t_k) \ge c_{t_k})$.  Thus we need to show that $\pi_1 \le \pi_2$.  This is so because of the fact that, for any $t$, $N_n(t)$ is stochastically non-decreasing in $\mu$, leading to
\beq
\pi_1 
= \P_{\mu_1}(N_n(t_1) \ge c_{t_1})
\le \P_{\mu_2}(N_n(t_1) \ge c_{t_1})
\le \P_{\mu_2}(N_n(t_2) \ge c_{t_2})
= \pi_2,
\eeq
where the last inequality is by construction of $t_2$ and $c_2$.
\end{proof}

Clearly, the oracle scan test has at least as much power as the oracle threshold test.  Interestingly, it does not have monotonic power in general, although it does under some natural assumptions on the base distribution.  

\begin{prp} \label{prp:scan_monotonicity}
Assume that $F$, as a distribution, is unimodal.  Then the oracle scan test has monotonic power in the shift.
\end{prp}

\begin{proof}
We stay with the setting and notation introduced in the proof of \prpref{threshold_monotonicity}.  Let $d \ge 0$ be smallest such that  
\beq
F[s_1 + d, t_1 + \mu_2 - \mu_1] = F[s_1, t_1].
\eeq
The fact that $F$, as a distribution, is unimodal implies that $d \le \mu_2 - \mu_1$.  
Now, under the null, by construction, 
\beq
\P_0(N_n[s_1 + d, t_1 + \mu_2 - \mu_1] \ge c_{s_1, t_1}) = \P_0(N_n[s_1, t_1] \ge c_{s_1, t_1}) \le \alpha.
\eeq
On the other hand, under $\P_{\mu_1}$, $N_n[s_1, t_1]$ is binomial with parameters $n$ and $q_1 := (1-\eps)F[s_1, t_1] + \eps F[s_1-\mu_1, t_1-\mu_1]$, while under $\P_{\mu_2}$, $N_n[s_1 + d, t_1 + \mu_2 - \mu_1]$ is binomial with parameters $n$ and 
\begin{align*}
q_2 
&:= (1-\eps)F[s_1 + d, t_1 + \mu_2 - \mu_1] + \eps F[s_1 + d -\mu_2, t_1 + \mu_2 - \mu_1 -\mu_2] \\
&= (1-\eps)F[s_1, t_1] + \eps F[s_1 + d -\mu_2, t_1 - \mu_1] \\
&\ge q_1,
\end{align*}
using the fact that $d \le \mu_2 - \mu_1$.
This explains the first inequality in the following derivation
\begin{align*}
\pi_1 
&= \P_{\mu_1}(N_n[s_1, t_1] \ge c_{s_1, t_1}) \\
&\le \P_{\mu_2}(N_n[s_1 + d, t_1 + \mu_2 - \mu_1] \ge c_{s_1, t_1}) \\
&\le \P_{\mu_2}(N_n[s_2, t_2] \ge c_{s_2, t_2}) 
= \pi_2,
\end{align*}
and the second inequality is by definition of $(s_2, t_2)$.
\end{proof}

\subsection{Performance bounds}
We now provide necessary and sufficient conditions for the the oracle threshold test and the oracle scan test to be fully powerful in the large-sample limit ($n\to\infty$).
We focus on the case where 
\beq
n \eps_n \to \infty, \qquad 
\sqrt{n} \eps_n \to 0,
\eeq
where the first condition implies that, under the alternative, the sample is indeed contaminated with probability tending to 1, while the second condition puts us in the regime corresponding to $\beta > 1/2$ under Ingster's parameterization \eqref{param}.
%We assume without loss of generality that $F(0) = 1/2$.

Our analysis below is based on the following simple result, which is an immediate consequence of Chebyshev's inequality and the central limit theorem.

\begin{lem}%[testing for a proportion] 
\label{lem:binom}
Suppose that we are testing $N \sim \Bin(n, p_n)$ versus $N \sim \Bin(n, q_n)$ where $p_n \le 1/2$ and $p_n \le q_n$, and consider the test at level $\alpha$ that rejects for large values of $N$ --- which is the most powerful test.  It is asymptotically powerful if $n (q_n - p_n)^2/q_n \to \infty$, while it is asymptotically powerless if $n (q_n - p_n)^2/p_n \to 0$.
\end{lem}

Using \lemref{binom}, we easily obtain performance guarantees for the oracle threshold test and the oracle scan test.

\begin{prp}%[oracle threshold test] 
\label{prp:threshold}
The oracle threshold test is powerful if there is a sequence of thresholds $(t_n)$ such that 
\beq\label{prp-threshold1}
n \eps_n \bar F(t_n - \mu_n) \to \infty, \quad \text{and} \quad
n \eps_n^2 \bar F(t_n - \mu_n)^2/\bar F(t_n) \to \infty.
\eeq
It is powerless if for any sequence of thresholds $(t_n)$
\beq\label{prp-threshold2}
n \eps_n^2 \bar F(t_n - \mu_n)^2/\bar F(t_n) \to 0.
\eeq
\end{prp}

\begin{proof}
Let $(t_n)$ denote a sequence of thresholds satisfying \eqref{prp-threshold1}, and define $p_n = \bar F(t_n)$ and $q_n = (1-\eps_n) \bar F(t_n) + \eps_n \bar F(t_n - \mu_n)$.
We know that $N_n(t_n) \sim \Bin(n, p_n)$ under the null and $N_n(t_n) \sim \Bin(n, q_n)$ under the alternative, with 
\begin{align*}
n (q_n - p_n)^2/q_n
&= \frac{n \eps_n^2 (\bar F(t_n - \mu_n) - \bar F(t_n))^2}{(1-\eps_n) \bar F(t_n) + \eps_n \bar F(t_n - \mu_n)}.
\end{align*}
If the second part of \eqref{prp-threshold1} holds, then necessarily $\bar F(t_n - \mu_n) \gg \bar F(t_n)$, since
\beq
n \eps_n^2 \bar F(t_n - \mu_n)^2/\bar F(t_n) 
= \big[n \eps_n^2 \bar F(t_n)\big] \big[\bar F(t_n - \mu_n)/\bar F(t_n)\big]^2
\le (n \eps_n^2) \big[\bar F(t_n - \mu_n)/\bar F(t_n)\big]^2,
\eeq
with $n \eps_n^2 = o(1)$ by assumption.
Hence,
\begin{align*}
n (q_n - p_n)^2/q_n
&\sim \frac{n \eps_n^2 \bar F(t_n - \mu_n)^2}{(1-\eps_n) \bar F(t_n) + \eps_n \bar F(t_n - \mu_n)} \\
&\asymp n \eps_n \bar F(t_n - \mu_n) \bigwedge n \eps_n^2 \bar F(t_n - \mu_n)^2/\bar F(t_n).
\end{align*}
Therefore, by \lemref{binom}, the sequence of tests $(\phi_{t_n})$ has full power in the limit when \eqref{prp-threshold1} holds.

Now let $(t_n)$ be any sequence of thresholds and consider the sequence of tests $(\phi_{t_n})$.  By \lemref{binom}, it has power $\alpha$ in the limit since 
\beq\label{prp-threshold-proof2}
n (q_n - p_n)^2/p_n 
%= n \eps_n^2 \big(\bar F(t_n - \mu_n) - \bar F(t_n)\big)^2/\bar F(t_n) 
\le n \eps_n^2 \bar F(t_n - \mu_n)^2/(1-\eps_n)\bar F(t_n)\to 0,
\eeq
where the convergence to $0$ comes from \eqref{prp-threshold2}.
\end{proof}

\begin{rem}\label{rem:prp-threshold1-simple}
Note that the first part of \eqref{prp-threshold1} may be replaced by 
\beq\label{prp-threshold1-simple}
n \bar F(t_n) \to \infty.
\eeq
This is because this and $n \eps_n^2 \bar F(t_n - \mu_n)^2/\bar F(t_n) \to \infty$ implies $n \eps_n \bar F(t_n - \mu_n) \to \infty$.
\end{rem}

\begin{prp}%[oracle scan test] 
\label{prp:scan}
The oracle scan test is powerful if there is a sequence of intervals $([s_n,t_n])$ such that 
\beq\label{prp-scan1}
n \eps_n F[s_n -\mu_n, t_n -\mu_n] \to \infty, \quad \text{and} \quad
n \eps_n^2 F[s_n -\mu_n, t_n -\mu_n]^2/F[s_n,t_n] \to \infty.
\eeq
It is powerless if for any sequence of intervals $([s_n,t_n])$
\beq
n \eps_n^2 F[s_n -\mu_n, t_n -\mu_n]^2/F[s_n,t_n] \to 0.
\eeq
\end{prp}

The proof is completely parallel to that of \prpref{threshold} and is omitted.

\subsection{Examples: generalized Gaussian models and more} \label{sec:generalized_Gaussian}
We look at a number of models and in each case derive the performance of the oracle threshold and oracle scan tests, and compare that with the performance of the likelihood ratio test.  

To place the results in line with the literature on the topic, we adopt Ingster's parameterization \eqref{param} for $\eps_n$, in fact a softer version of that
\beq\label{eps}
\eps = \eps_n \sim n^{-\beta},
\eeq
for some fixed $\beta$.  The parameterization of $\mu = \mu_n$ will depend on on the model.

To further simplify matters, we assume throughout that 
\beq\label{varphi}
\log \bar F(x) \sim -\varphi(x),
\eeq
where $\varphi(x)$ is continuous and strictly increasing for $x$ large enough.   
In that case, in view of \remref{prp-threshold1-simple}, we note that \eqref{prp-threshold1} is satisfied when 
\beq\label{prp-threshold1-1}
\log n - \varphi(t_n) \to \infty, \qquad 
(1 - 2\beta)\log n + \varphi(t_n) - 2\varphi(t_n - \mu_n) \to \infty.
\eeq

\subsubsection{Extended generalized Gaussian}
This class of models is defined by the property that $\varphi$ satisfies\footnote{ It is tempting to consider a more general condition where there is a function $\omega$ on $\bbR_+$ such that $\lim_{t\to\infty}\varphi(u t)/\varphi(t) \to \omega(u)$ for all $u \ge 0$.  However, as long as $\omega$ is not constant (equal to zero in that case), it can easily be shown that $\omega(u) = u^a$ for some $a > 0$.}
\beq \label{aggcondition}
\varphi(u t)/\varphi(t) \to u^a, \quad t \to \infty, \quad \forall u \ge 0.
\eeq
Here $a > 0$ parameterizes this class of models.
This covers the generalized Gaussian models, which are often used as benchmarks in this line of work.  It also covers the case where $\varphi(t) \sim t^a (\log t)^b$ where $b \in \bbR$ is arbitrary.

For $a > 1$, define 
\beq \label{gamma>1}
\rho_a(\beta) = \begin{cases} 
(2^{1/(a -1)} - 1)^{a -1} (\beta - 1/2), & 1/2 < \beta < 1 - 2^{-a/(a -1)}, \\ 
(1 - (1-\beta)^{1/a})^{a}, & 1 - 2^{-a/(a -1)} \le \beta < 1.
\end{cases}
\eeq
For $a \le 1$, define
\beq \label{gamma<1}
\rho_a(\beta) = 2 (\beta - 1/2).
\eeq
In addition to \eqref{eps}, assume that 
\beq\label{mu1}
\mu = \mu_n \text{ satisfies } \varphi(\mu_n) \sim r \log n, \text{ with $r \ge 0$ fixed}.
\eeq

\begin{prp}\label{prp:generalized_gaussian}
The curve $r = \rho_a(\beta)$ in the $(\beta,r)$ plane is the detection boundary that the oracle threshold test achieves.
\end{prp}

\begin{proof}
We focus on proving that the oracle threshold test achieves that boundary.  A simple inspection of the arguments reveal that they are tight, so that this is the precise detection boundary that the test achieves.  (See the proof of \prpref{oracle_threshold_power_law} for an example.)

We divide the proof into several cases.

\medskip\noindent
{\em Case 1: $a > 1$.}  
Define $b = 2^{-1/(a - 1)}$ and note that $0 < b < 1$.  

\medskip\noindent
{\em Case 1.1: $1/2 < \beta < 1 - b^{a}$ and $r > \rho_a(\beta)$.}
Under these conditions, $\beta < 1/2 + r(1/b - 1)^{-(a-1)}$, and in particular there is $\eta > 0$ such that 
\beq\label{beta-ub1}
1 - 2\beta \ge -2r(1/b - 1)^{-(a-1)} + \eta.
\eeq
Setting $t_n = (1 - b)^{-1}\mu_n$, by \eqref{aggcondition} and \eqref{mu1}, we have the following
\begin{align}
\varphi(t_n - \mu_n) 
&= \big(r b^a/(1-b)^a + o(1)\big) \log n, \\
\varphi(t_n)
&= \big(r/(1 - b)^a + o(1)\big) \log n.
\end{align}
By \prpref{threshold_monotonicity} we may focus on $r$ small enough that $r/(1-b)^a < 1$.  This is possible because $\rho_a(\beta) < (1-b)^a$ when $\beta < 1 - b^a$, which we assume here.  (This can be easily verified using the definition of $b$.)
Assuming that $r$ is as such, the first part of \eqref{prp-threshold1-1} is satisfied.  
For the second part, with \eqref{beta-ub1}, we have
\begin{align*}
&(1 - 2\beta)\log n - 2\varphi(t_n - \mu_n) + \varphi(t_n) \\
&\ge \big[-2r (1/b - 1)^{-(a-1)} +\eta -2 r b^a/(1-b)^a + r/(1-b)^a +o(1) \big] \log n \\
&=[\eta + o(1)] \log n \to \infty,
\end{align*}
using the definition of $b$ and simplifying.
Thus the second part of \eqref{prp-threshold1-1} is also satisfied and the oracle threshold test is powerful.

\medskip\noindent
{\em Case 1.2: $1 - b^{a} \le \beta < 1$ and $r > \rho_a(\beta)$.}
Under these conditions, we have $1 - \beta > (1 - r^{1/a})^a$,  and in particular there is $\eta > 0$ such that 
\beq\label{beta-ub2}
1 - \beta - \eta \ge (1 - r^{1/a})^a \ge ((1 - \eta)^{1/a} - r^{1/a})^a. 
\eeq
Set $t_n = (\frac{1}{r}(1  - \eta))^{1/a}\mu_n$, we have the following
\begin{align}
\varphi(t_n - \mu_n) 
&= \big((1 - \eta)^{1/a} - r^{1/a})^a + o(1)\big) \log n, \\
\varphi(t_n)
&= (1 - \eta + o(1)) \log n.
\end{align}
By looking at the speed of $\varphi(t_n)$, the first part of \eqref{prp-threshold1-1} is satisfied immediately. For the second part, with \eqref{beta-ub1}, we have
\begin{align*}
&(1 - 2\beta)\log n - 2\varphi(t_n - \mu_n) + \varphi(t_n) \\
&= (1 - 2\beta)\log n - 2\big((1 - \eta)^{1/a} - r^{1/a})^a + o(1)\big) \log n + (1 - \eta + o(1))\log n \\
&= 2 \big[1 - \beta -\eta/2 - ((1 - \eta)^{1/a} - r^{1/a})^a + o(1)\big] \log n
\to \infty.
\end{align*}
Thus the second part of \eqref{prp-threshold1-1} is also satisfied and the oracle threshold test is powerful.

\medskip\noindent
{\em Case 2: $a \le 1$.}  
By \prpref{threshold_monotonicity} we may restrict attention to the case where $2 \beta - 1 < r < 1$.  
Here we set $t_n = \mu_n$.
Then the first part in \eqref{prp-threshold1-1} is clearly satisfied. For the second part, notice that
\begin{align*}
&(1 - 2\beta)\log n - 2\varphi(t_n - \mu_n) + \varphi(t_n) \\
&=(1 - 2\beta)\log n + (r + o(1)) \log n \\
&=[1 - 2\beta + r + o(1)] \log n \to \infty. \qedhere
\end{align*}
\end{proof}

Thus, although the conditions are much more general here, the detection boundary is the same as in the corresponding generalized Gaussian model and, moreover, the oracle threshold test achieves that boundary.

\begin{rem}[max test]
In this class of models, it can be shown that the max test  achieves the detection boundary over the upper range, meaning when $\beta \ge 1 - 2^{-a/(a -1)}$.  In fact, $\rho^{\rm max}(\beta) := (1 - (1-\beta)^{1/a})^{a}$ defines the detection boundary for the max test.
\end{rem}

\subsubsection{Other models}
In the next few classes of models, $\varphi$ satisfies
\beq\label{omega}
\frac{\varphi^{-1}(t) - \varphi^{-1}(v t)}{\lambda(t)} \to \omega(v), \quad t \to \infty, \quad \forall v \in (0,1].
\eeq
for some functions $\lambda$ and $\omega$, with the latter being non-increasing, continuous, and such that $\omega(1) = 0$.  This is actually also the case when $\varphi(t) \sim t^a (\log t)^b$ with $a > 0$ and $b \in \bbR$, with $\lambda(t) = t^{1/a} (\log t)^{-b/a}$ and $\omega(v) = (1 - v^{1/a})/a^{b/a}$.

Define
\beq\label{rho1}
\rho(\beta) = \inf_{0 < h < 1 - \beta} \, \big[\omega(h) - \omega(2\beta -1 + 2h)\big].
\eeq
In addition to \eqref{eps}, assume that 
\beq
\mu = \mu_n \sim r \lambda(\log n), \quad r \ge 0 \text{ fixed}.
\eeq

\begin{prp}\label{prp:generalized_gaussian_other}
The curve $r = \rho(\beta)$ in the $(\beta,r)$ plane is the detection boundary that the oracle threshold test achieves.
\end{prp}

\begin{proof}
We focus on proving that the oracle threshold test achieves that boundary.  

Since $\omega(v)$ is continuous, we may define 
\beq
h^* = \argmin_{0 \le h \le 1 - \beta} \, \big[\omega(h) - \omega(2\beta -1 + 2h)\big].
\eeq
We focus on the case where $h^* < 1 -\beta$.  In the case where $h^* = 1 -\beta$, the max test is powerful (\remref{max_test2}), and therefore so is the oracle threshold test.  
By \prpref{threshold_monotonicity} we may focus on the case where $r < \omega(h^*)$.  
With these assumptions and the fact that $\omega(h^*) - \omega(2\beta -1 + 2h^*) = \rho(\beta) < r$, there is $\eta > 0$ be such that 
\beq\label{proof_generalized_gaussian_other1}
2\beta - 1 + 2h^* + 2\eta < 1,
\eeq
and
\beq
\omega(h^*) - \omega(2\beta -1 + 2h^* +\eta) 
< r 
< \omega(h^*) - \omega(2\beta -1 + 2h^* +2\eta).
\eeq

Define $t_n := \mu_n +\varphi^{-1}(h^*\log n)$.  
Using \eqref{omega} multiple times, for $n$ sufficiently large, we have the following
\begin{align*}
\mu_n 
&= (r +o(1)) \lambda(\log n)\\
&\le [\omega(h^*) - \omega(2\beta - 1+ 2h^*+ 2\eta)]\lambda(\log n)\\
&= \varphi^{-1}(\log n) - \varphi^{-1}(h^*\log n) - \varphi^{-1}(\log n) + \varphi^{-1}((2\beta - 1+ 2h^* +2\eta)\log n)\\
&= \varphi^{-1}((2\beta - 1+ 2h^* +2\eta)\log n)  -\varphi^{-1}(h^*\log n).
\end{align*}
Hence, eventually, $t_n \le \varphi^{-1}((2\beta - 1+ 2h^* +2\eta)\log n)$, implying that
\begin{align*}
\log n - \varphi(t_n) 
= \log n -(2\beta - 1+ 2h^* +2\eta)\log n
= [1 -(2\beta -2 +2h^* +2\eta)] \log n
\to \infty,
\end{align*}
using \eqref{proof_generalized_gaussian_other1}.
Thus the first part of \eqref{prp-threshold1-1} is satisfied.

Similarly, for $n$ sufficiently large,
\begin{align*}
\mu_n 
&= (r +o(1)) \lambda(\log n)\\
&\ge [\omega(h^*) - \omega(2\beta - 1+ 2h^*+ \eta)]\lambda(\log n)\\
&= \varphi^{-1}((2\beta - 1+ 2h^* +\eta)\log n)  -\varphi^{-1}(h^*\log n),
\end{align*}
so that, eventually, $t_n \ge \varphi^{-1}((2\beta - 1+ 2h^*+\eta)\log n)$, implying that
\begin{align*}
&(1 - 2\beta)\log n -2\varphi(t_n - \mu_n) +\varphi(t_n) \\
&\ge (1 - 2\beta)\log n -2 h^* \log n + (2\beta - 1 + 2h^* + \eta)\log n\\
&= \eta \log n \to \infty.
\end{align*}
Thus the second part of \eqref{prp-threshold1-1} is satisfied.
\end{proof}

\begin{rem}[max test] \label{rem:max_test2}
In the present situation, it can be shown that $\rho^{\rm max}(\beta) := \omega(1-\beta)$ defines the detection boundary for the max test.
\end{rem}

\subsubsection{Extended generalized Gumbel}
This class of models is defined by $\varphi(t) = \exp(t^a)$ for some $a > 0$, which satisfies \eqref{omega} with $\lambda(t) = \tfrac1a (\log t)^{1/a-1}$ and $\omega(v) = \log(1/v)$.
In this case, 
\beq
\mu = \mu_n \sim \frac{r}a (\log \log n)^{1/a-1},
\eeq
and the detection boundary is given by $r = - \log(1-\beta)$.
Note that, at the detection boundary, $\mu_n \to \infty$ when $a > 1$; that $\mu_n \asymp 1$ when $a = 1$; and $\mu_n \to 0$ when $a < 1$.

\subsubsection{Extended generalized Gumbel}
This class of models is defined by $\varphi(t) = \exp((\log t)^a)$ for some $a > 1$, which satisfies \eqref{omega} with $\lambda(t) = \tfrac1a (\log t)^{1/a-1} \exp((\log t)^{1/a})$ and $\omega(v) = \log(1/v)$.
In this case, 
\beq
\mu = \mu_n \sim \frac{r}a (\log \log n)^{1/a-1} \exp((\log \log n)^{1/a}),
\eeq
and the detection boundary is given by $r = - \log(1-\beta)$ as in the previous class of models (since $\omega$ is the same).

\begin{rem}[max test] \label{rem:max_test3}
Based on \remref{max_test2}, in the last two classes of models, the max test achieves the detection boundary over the whole $\beta$ range.  The same is true, more generally, when the infimum in \eqref{rho1} is at $h = 1-\beta$.
\end{rem}
 
\subsection{Examples: power-law models and more} \label{sec:power_law}
In the next few classes of models, $F$ satisfies
\beq\label{F1}
\log(F(t+v) - F(t)) \sim -\lambda(t), \quad  t \to \infty, \quad \forall v \ge 0,
\eeq
for some function $\lambda$ which is increasing eventually and such that $\lambda(t) \to \infty$ as $t \to \infty$.  
This includes models where 
\beq\label{F2}
\bar F(t) \propto t^{-a} (\log t)^b (1 + o(1/t)), \quad t \to \infty,
\eeq 
with $a > 0$ and $b \in \bbR$, in which case \eqref{F1} holds with $\lambda(t) = (a+1) \log t$.  
It also includes models where $\bar F(t) \propto (\log t)^{-a}(1 + o(1/t\log t))$, with $a > 0$, in which case \eqref{F1} holds with $\lambda(t) = \log t$, as well as other distribution with even slower decay.

In addition to \eqref{eps}, assume that
\beq\label{mu2}
\mu = \mu_n \quad \text{satisfies} \quad \lambda(\mu_n) \sim r \log n, \quad r \ge 0 \text{ fixed}.
\eeq

\begin{prp}\label{prp:power_law}
The curve $r = \rho(\beta) := 2\beta -1$ in the $(\beta,r)$ plane is the detection boundary that the oracle scan test achieves.
\end{prp}

\begin{proof}
We focus on proving that the oracle scan test achieves that boundary.  
  
Fix $r$ such that $r > 2\beta - 1$.  
Consider the interval $[s_n, t_n]$ with $s_n := \mu_n$ and $t_n := \mu_n + v$, where $v > 0$ is such that $F[0,v] > 0$.  
We need to verify that \eqref{prp-scan1} holds.
On the one hand, we have 
\beq
n \eps_n F[s_n -\mu_n, t_n -\mu_n] 
= n^{1 -\beta} F[0, v]
\to \infty,
\eeq
because $\beta < 1$ by assumption.
So the first part of \eqref{prp-scan1} holds.
On the other hand, 
%taking the logarithm,
\begin{align*}
&n \eps_n^2 F[s_n -\mu_n, t_n -\mu_n]^2/F[s_n,t_n]
= n^{1-2\beta} F[0,v]^2/n^{r+o(1)}
= n^{r + 1 -2\beta + o(1)} \to \infty,
\end{align*}
since $r > 2\beta - 1$.
So the second part of \eqref{prp-scan1} holds.
\end{proof}

We now show that threshold tests are suboptimal in the main class of models satisfying \eqref{F1}, namely \eqref{F2}.  
(The same happens to be true in other models with fat tails satisfying \eqref{F1}.)  
This is the main motivation for considering scan tests.

\begin{prp} \label{prp:oracle_threshold_power_law}
In a model satisfying \eqref{F2}, and with the same parameterization \eqref{mu2}, the curve $r = (1+1/a)(2\beta -1)$ in the $(\beta,r)$ plane is the detection boundary that the oracle threshold test achieves.
\end{prp}

\begin{proof}
We first prove that the oracle threshold test achieves this detection boundary.  By \prpref{threshold_monotonicity} we may assume that $r < 1+1/a$.  Therefore, fix $r$ such that $(1+1/a) (2\beta -1) < r < 1+1/a$.  
Set the threshold $t_n = \mu_n + v$, where $v$ is such that $\bar F(v) > 0$.  
We need to verify that \eqref{prp-threshold1} holds, and we do so via \remref{prp-threshold1-simple}.
Note that $t_n \sim \mu_n = n^{r/(a+1) + o(1)}$.
In particular,
\beq
n \bar F(t_n) 
\sim n \mu_n^{-a} (\log \mu_n)^b
= n^{1 - a r/(a+1) + o(1)}
\to \infty,
\eeq
and, by the same token,
\beq
n \eps_n^2 \bar F(t_n - \mu_n)^2/\bar F(t_n)
\sim n^{1-2\beta} n^{- a r/(a+1) + o(1)}
= n^{1-2\beta - a r/(a+1) + o(1)}
\to \infty.
\eeq

We now turn to proving that this is the statement boundary is the best that the oracle threshold test can hope for.  
For this, fix $r < (1+1/a) (2\beta -1)$. 
We need to verify \eqref{prp-threshold2}.
Suppose for contradiction that there is a sequence of thresholds, $(t_n)$, such that \eqref{prp-threshold2} does not hold.  By extracting a subsequence if needed, we may assume that
\beq\label{proof_oracle_threshold_power_law1}
n \eps_n^2 \bar F(t_n - \mu_n)^2/\bar F(t_n) \to \lambda \in (0, \infty].
\eeq
First, suppose that $\liminf t_n/\mu_n < \infty$.
Extracting a subsequence if needed, we may assume that $t_n = O(\mu_n)$.  In that case, we have
\beq
n \eps_n^2 \bar F(t_n - \mu_n)^2/\bar F(t_n)
\le n \eps_n^2/\bar F(t_n)
\le n^{1 -2\beta + o(1)} \mu_n^{-a + o(1)}
= n^{1 -2\beta - a r/(a+1) + o(1)}
\to 0.
\eeq
Since this contradicts \eqref{proof_oracle_threshold_power_law1}, we must have $\liminf t_n/\mu_n = \infty$, meaning that $t_n \gg \mu_n$.  In that case, we have $\bar F(t_n - \mu_n) \sim \bar F(t_n)$, implying that
\beq
n \eps_n^2 \bar F(t_n - \mu_n)^2/\bar F(t_n)
\sim n \eps_n^2 \bar F(t_n)
\le n \eps_n^2
\to 0.
\eeq
This also contradicts \eqref{proof_oracle_threshold_power_law1}.
Since there is no other option, it must be that \eqref{proof_oracle_threshold_power_law1} cannot hold.
We conclude that, indeed, \eqref{prp-threshold2} holds for any sequence of thresholds.
\end{proof}

\section{Scan tests} \label{sec:scan}
In this section, we study the scan tests \eqref{HC_scan} and \eqref{BJ_scan}, and show that both of them do as well as the oracle scan test, at least to first-order in the asymptote where $n \to \infty$ and under the various parameterizations used in the previous section.
We refer to \eqref{HC_scan} as the Stouffer scan test, as it is constructed as Stouffer's combination test \citep{stouffer1949american}; while we refer to \eqref{BJ_scan} as the Tippett scan test, for similar reasons \citep{tippett1931methods}.

\subsection{Stouffer scan test}
We study the Stouffer scan test \eqref{HC_scan}.  The main work goes into controlling this statistic under the null hypothesis.  The limiting distribution of higher criticism can be derived from \citep{jaeschke1979asymptotic} and the limiting distributions of some variants of the scan statistic are known under other models \citep{kabluchko2011extremes, sharpnack2016exact}.
We will not pursue such a fine result here, but contend ourselves with a relatively rough upper bound.

\begin{lem} \label{lem:HC_scan}
Given observations $x_1, \dots, x_n$, the maximum in \eqref{HC_scan} is attained at some $(s,t) = (x_i,x_j)$.
\end{lem}

\begin{proof}
Define
\beq\label{R_n}
R_n(s,t) := \frac{N_n[s,t] - n F[s,t]}{\sqrt{n F[s,t] (1 -F[s,t]) + 1}}.
\eeq
Let $x_{(1)} \le \cdots \le x_{(n)}$ denote the ordered observations, and set $x_{(0)} = -\infty$ and $x_{(n+1)} = \infty$.  It suffices to show that, for any $1 \le i \le j \le n$ and any $(s,t)$ such that $x_{(i-1)} < s \le x_{(i)}$ and $x_{(j)} \le t < x_{(j+1)}$, in addition to $F[s,t] \le 1/2$, we have $R_n(x_{(i)}, x_{(j)}) \ge R_n(s,t)$.  The crucial observation is that $N_n[s,t] = N_n[x_{(i)}, x_{(j)}]$ while $F[x_{(i)}, x_{(j)}] \le F[s,t]$.  

It is thus enough to show that the function $p \mapsto (a - p)/(p(1-p) + b)^{1/2}$ is decreasing over $[0,1/2]$ for any $a, b \ge 0$.  This is so since this function has derivative
$- (a (1 -2 p) + 2 b + p)/(p(1-p) + b)^{3/2}$.
\end{proof}

\begin{thm} \label{thm:HC_scan0}
With $S_n$ defined as the statistic \eqref{HC_scan}, we have 
\beq
\P_0(S_n \ge 3\log n) \to 0.
\eeq
\end{thm}

\begin{proof}
We place ourselves under the null hypothesis.
Recall the definition of $R_n$ in \eqref{R_n}.
By \lemref{HC_scan} and the fact that $R_n(X_i, X_i) = 1$ for all $i$, if $S_n \ge 3 \log n$ necessarily $S_n = S_n^* := \max_{i \ne j} R_n(X_i, X_j)$.  For any $i \ne j$, we have
\beq
R_n(X_i, X_j) \le 2 + S_{i,j},
\eeq
with
\beq\label{Nij}
S_{i,j} := \frac{N_{i,j} -2 - (n-2) p_{i,j}}{\sqrt{(n-2) p_{i,j}(1-p_{i,j}) + 1}}, \quad N_{i,j} := N_n[X_i, X_j], \quad p_{i,j} := F[X_i, X_j].
\eeq
The point of this reorganizing is that, given $(X_i, X_j)$, $N_{i,j}-2 \sim \Bin(n-2, p_{i,j})$,
and an application of Bernstein's inequality gives
\beq
\P_0(S_{i,j} \ge s \mid X_i, X_j) 
\le \exp\bigg(-\frac{s^2 b_{i,j}^2/2}{b_{i,j}^2 +s b_{i,j}/3}\bigg)
\le \exp\bigg(-\frac{s^2/2}{1 +s/3}\bigg)
\le \exp(-s), \quad \forall s \ge 6,
\eeq
because $b_{i,j} := \sqrt{(n - 2) p_{i,j} (1-p_{i,j}) + 1} \ge 1$.
Thus, with the union bound, as $n \to \infty$, we have
\begin{align*}
\P_0(S_n \ge 3 \log n) 
&= \P_0(\exists i \ne j : S_{i,j} + 2 \ge 3 \log n) \\
&\le \sum_{i < j} \P_0(S_{i,j} \ge 3 \log n - 2)
\le n^2 \exp(-3 \log n +2)
\to 0,
\end{align*}
which proves the statement.
\end{proof}

With \thmref{HC_scan0}, one obtains the following performance bound for the Stouffer scan test.

\begin{cor}
The Stouffer scan test is powerful if there is a sequence of intervals $([s_n,t_n])$ such that 
\beq\label{cor-stouffer1}
n \eps_n F[s_n -\mu_n, t_n -\mu_n] \gg \log n, \quad \text{and} \quad
n \eps_n^2 F[s_n -\mu_n, t_n -\mu_n]^2/F[s_n,t_n] \gg (\log n)^2.
\eeq
\end{cor}

\begin{proof}
By \thmref{HC_scan0}, the Stouffer scan test at level $\alpha$ is at least as powerful as the test $\{S_n \ge 3 \log n\}$, eventually.  Now, under the alternative, this test is powerful if we can prove that $p_n := F[s_n, t_n] \le 1/2$ and $R_n(s_n, t_n) \ge 3 \log n$.  
Define $p'_n := F[s_n - \mu_n, t_n - \mu_n]$ and $q_n := (1-\eps_n) p_n + \eps_n p'_n$, so that \eqref{cor-stouffer1} can be expressed as
\beq
n \eps_n p'_n \gg log n \to \infty, \quad \text{and} \quad
n \eps_n^2 {p'_n}^2/p_n \gg (log n)^2 \to \infty.
\eeq  
That $p_n \le 1/2$ is true, eventually, comes from the fact that
\beq
\infty \gets n \eps_n^2 {p'_n}^2/p_n \le n \eps_n^2/p_n,
\eeq
with $n \eps_n^2 \to 0$ by assumption, so that necessarily $p_n \to 0$.  Note that this implies that $q_n \to 0$ also.

Given that $N_n[s_n, t_n]$ is binomial with parameters $n$ and $q_n$, with $n q_n \ge n p'_n \to \infty$ by the first part of \eqref{cor-stouffer1}, we have $N_n[s_n, t_n] = n q_n + O_P(\sqrt{n q_n (1-q_n)})$, and so 
\beq
R_n(s_n, t_n) 
= \frac{n \eps_n (p'_n - p_n) + O_P(\sqrt{n q_n (1-q_n)})}{\sqrt{n p_n (1-p_n)+ 1}}
\sim \frac{n \eps_n p'_n + O_P(\sqrt{n q_n})}{\sqrt{n p_n + 1}},
\eeq
since $p'_n \gg p_n$, by the fact that
\beq
\infty \gets n \eps_n^2 {p'_n}^2/p_n = n \eps_n^2 p_n (p'_n/p_n)^2 = o(p'_n/p_n)^2.
\eeq
In addition, the same conditions imply
\beq
\frac{n \eps_n p'_n}{\sqrt{n q_n}} \asymp \sqrt{n \eps_n^2 {p'_n}^2/p_n} \bigvee n \eps_n p'_n \to \infty,
\eeq
so that 
\beq
R_n(s_n, t_n) 
\sim_P n \eps_n p'_n/\sqrt{n p_n + 1}
\asymp_P \sqrt{n \eps_n^2 {p'_n}^2/p_n} \bigvee n \eps_n p'_n
\gg \log n.
\eeq
We conclude that $R_n(s_n, t_n) \ge 3 \log n$ holds with probability tending to 1.
\end{proof}

With this performance bound, it is straightforward to verify that the Stouffer scan test performs as well as the oracle scan test to first order, at least in the context of the parameterization used in the models studied in \secref{generalized_Gaussian} and \secref{power_law}.  This comes from the fact that, in context of these sections, the quantity appearing in \eqref{cor-stouffer1} increases as a (fixed) positive power of $n$ under the alternative.   
We formalize this into the following statement, left without formal proof.

\begin{cor}\label{cor:stouffer}
The Stouffer scan test achieves the oracle scan detection boundary in all the settings considered in \secref{generalized_Gaussian} and \secref{power_law}.
\end{cor}

\subsection{Tippett scan test}
We study the Tippett scan test \eqref{BJ_scan}, which we denote by $T_n$.  We control this statistic under the null hypothesis by a simple application of the union bound.  A more refined control seems possible in view of \cite{moscovich2016exact}, where the limiting distribution of \eqref{BJ} is obtained.

\begin{prp} \label{prp:BJ_scan0}
With $T_n$ defined as the statistic \eqref{BJ_scan}, we have 
\beq
\P_0(T_n \le 1/n^3) \to 0.
\eeq
\end{prp}

\begin{proof}
Under the null, each $P_{i,j}$ is uniformly distributed in $[0,1]$.  Thus the union bound gives
\beq
\P_0(T_n \le 1/n^3)
\le n^2 \P_0(P_{i,j} \le 1/n^3)
= n^2/n^3 = 1/n \to 0,
\eeq
which concludes the proof.
\end{proof}

Thus most of the work goes into controlling the statistic under the alternative.  We do so by bounding the Tippett scan statistic by an expression that resembles that of the Stouffer scan statistic.  We make use of the following simple concentration bound.\footnote{\quad Many things are known about the beta distribution and order statistics in general, but we could not immediately find such a simple bound.}

\begin{lem}\label{lem:tippett}
For $k \in [n]$,
\beq
\Beta(u; k, n-k+1) \le \exp\bigg(- \frac{(k -n u)^2/2}{nu(1-u) + (k-nu)/3}\bigg), \quad 0 \le u \le k/n.
\eeq 
\end{lem}

\begin{proof}
Let $U_{k:n}$ denote the $k$-th order statistic of an iid sample of size $n$ from the uniform distribution on $[0,1]$.
For $u \in [0,1]$ such that $nu \le k$, we have 
\begin{align*}
\Beta(u; k, n-k+1)
= \P(U_{k:n} \le u)
= \P(\Bin(n, u) \ge k),
\end{align*}
and we conclude with an application of Bernstein's inequality.
\end{proof}

\begin{prp}
The Tippett scan test is powerful if there is a sequence of intervals $([s_n,t_n])$ such that
\beq\label{prp-tippett1}
n \eps_n F[s_n -\mu_n, t_n -\mu_n] \gg \sqrt{\log n}, \qquad
n \eps_n^2 F[s_n -\mu_n, t_n -\mu_n]^2/F[s_n,t_n] \gg \log n.
\eeq
\end{prp}

\begin{proof}
Recall that $T_n = \min_{i < j} P_{i,j}$ and the expression of $P_{i,j}$.  Thus an application of \lemref{tippett} gives
\beq
T_n \le 1/n^3 \iff 
\max_{i < j} \frac{(j-i - V_{i,j})_+^2/2}{n V_{i,j} (1-V_{i,j}) + (j-i - V_{i,j})_+/3} \ge 3 \log n,
\eeq
where $V_{i,j} := U_{(j)} - U_{(i)}$, after taking a logarithm.
Moreover, $V_{i,j} = F[X_{(n-j+1)}, X_{(n-i+1)}]$ and $j-i = N_n[X_{(n-j+1)}, X_{(n-i+1)}] - 1$, yielding
\beq
T_n \le 1/n^3 \iff 
\max_{i\ne j} \frac{(N_{i,j} -1 -n p_{i,j})_+^2}{n p_{i,j} (1-p_{i,j}) +(N_{i,j} -1 -n p_{i,j})_+} \ge 6 \log n,
\eeq
with the notation of \eqref{Nij}.
The latter inequality holds when there is $i \ne j$ such that
\beq
N_{i,j} -1 -n p_{i,j} \ge 12 \log n 
\quad \text{and} \quad
\frac{N_{i,j} -1 -n p_{i,j}}{\sqrt{n p_{i,j} (1-p_{i,j})}} \ge \sqrt{12 \log n},
\eeq
which is the case when
\beq\label{tippett-proof1}
n p_{i,j} \ge \sqrt{12 \log n}
\quad \text{and} \quad
\frac{N_{i,j} -1 -n p_{i,j}}{\sqrt{n p_{i,j} (1-p_{i,j})}} \ge \sqrt{12 \log n}.
\eeq

Let $(s_n, t_n)$ be as in the statement and let $(i,j)$ be such that $U_{(i)} \le s < U_{(i+1)}$ and $U_{(j-1)} < t \le U_{(j)}$.
By construction, $p_{i,j} \ge F[s_n, t_n]$, so that the first part of \eqref{prp-tippett1} implies that the first part of \eqref{tippett-proof1} holds eventually.  
We also have $N_{i,j} \ge N_n[s_n,t_n] - 2$, so that
\beq
\frac{N_{i,j} -1 -n p_{i,j}}{\sqrt{n p_{i,j} (1-p_{i,j})}} \ge \frac{N_n[s_n,t_n] -3 -n F[s_n,t_n]}{\sqrt{n F[s_n,t_n] (1-F[s_n,t_n])}},
\eeq
and the quantity on the RHS is controlled using the second part \eqref{prp-tippett1} exactly as in the proof of \prpref{scan}.
\end{proof}

Here too, these results make it straightforward to verify that the Tippett scan test performs as well as the oracle scan test (to first order) in the models and regimes seen earlier, leading us to state the following (left without a formal proof).

\begin{cor}\label{cor:tippett}
The Tippett scan test achieves the oracle scan detection boundary in all the settings considered in \secref{generalized_Gaussian} and \secref{power_law}.
\end{cor}

\section{Numerical experiments} \label{sec:numerics}

We performed small-scale numerical experiments to probe our theory.  We worked with Student t-distributions with varying numbers of degrees of freedom, ${\rm df} = 0.5, 1, 2, 5\}$.
Recall that the Student t-distribution with $k$ degrees of freedom has density $\propto (1 + x/k)^{-(k+1)/2}$.
We considered three different scenarios with varying sparsity exponents, $\beta = 0.6, 0.7, 0.8$.  
The sample size was set to $n = 30,000$. 
We compared the higher criticism test, the Berk-Jones test, the Stouffer scan test, and the Tippett scan test in each of these settings.  We repeat each setting 200 times. 
See \figref{6}, \figref{7}, and \figref{8}.

As the theory predicts, We can check that when the number of degrees of freedom is smaller, implying that the base distribution has fatter tails, the scan procedures dominate the threshold procedure.  The threshold procedures become dominant as the tails become lighter.  This is so at this particular sample size as, in principle, our theory indicates that with a larger sample size, the scan procedures would still dominate.  The transition from powerless to powerful takes place at a larger effect size than predicted by the theory, which is also explain by the limited sample size.\footnote{The scan tests have computational complexity of order $O(n^2)$, which has limited the scale of our experiments.}

\section{Discussion} \label{sec:discussion}
While scan tests are commonly used in a number of detection problems, threshold tests are almost exclusively used in multiple testing situations.  The main purpose of our work here was to reveal that scan tests can improve on threshold tests in somewhat standard multiple testing settings, particularly when the null distribution ($F$ in the paper) has heavy tails.  

\paragraph{Likelihood ratio performance bounds}
Given our main objective, it was more natural to consider oracle-type performance bounds rather than using the likelihood ratio performance as benchmark.  
We can say nonetheless that, for representative models, the oracle threshold boundaries stated in \prpref{generalized_gaussian} and \prpref{generalized_gaussian_other} match those of the likelihood ratio test --- for example, this is true of generalized Gaussian models where $F$ has density of the form $f(t) \propto \exp(-|t|^a)$ for some $a > 0$. 
The same is true of the oracle scan boundary stated in \prpref{power_law} --- for example, this is true of power law models where $F$ has density of the form $f(t) \propto (1+|t|^a)^{-1}$ for some $a > 0$.

\paragraph{Nonparametric approaches}
\citet{arias2017distribution} consider the situation where the null distribution, $F$, is symmetric about 0 but otherwise unknown.  They suggest two tests for symmetry: the CUSUM sign test and the tail-run test, which are meant to be the nonparametric equivalent of the higher criticism test and the tail-run sign test, respectively.  Back-of-the-envelope calculations seem to indicate that these nonparametric tests achieve the same detection boundaries as their parametric counterparts in all the settings considered here.  

\begin{figure}[h!]
\centering
\includegraphics[scale=0.6]{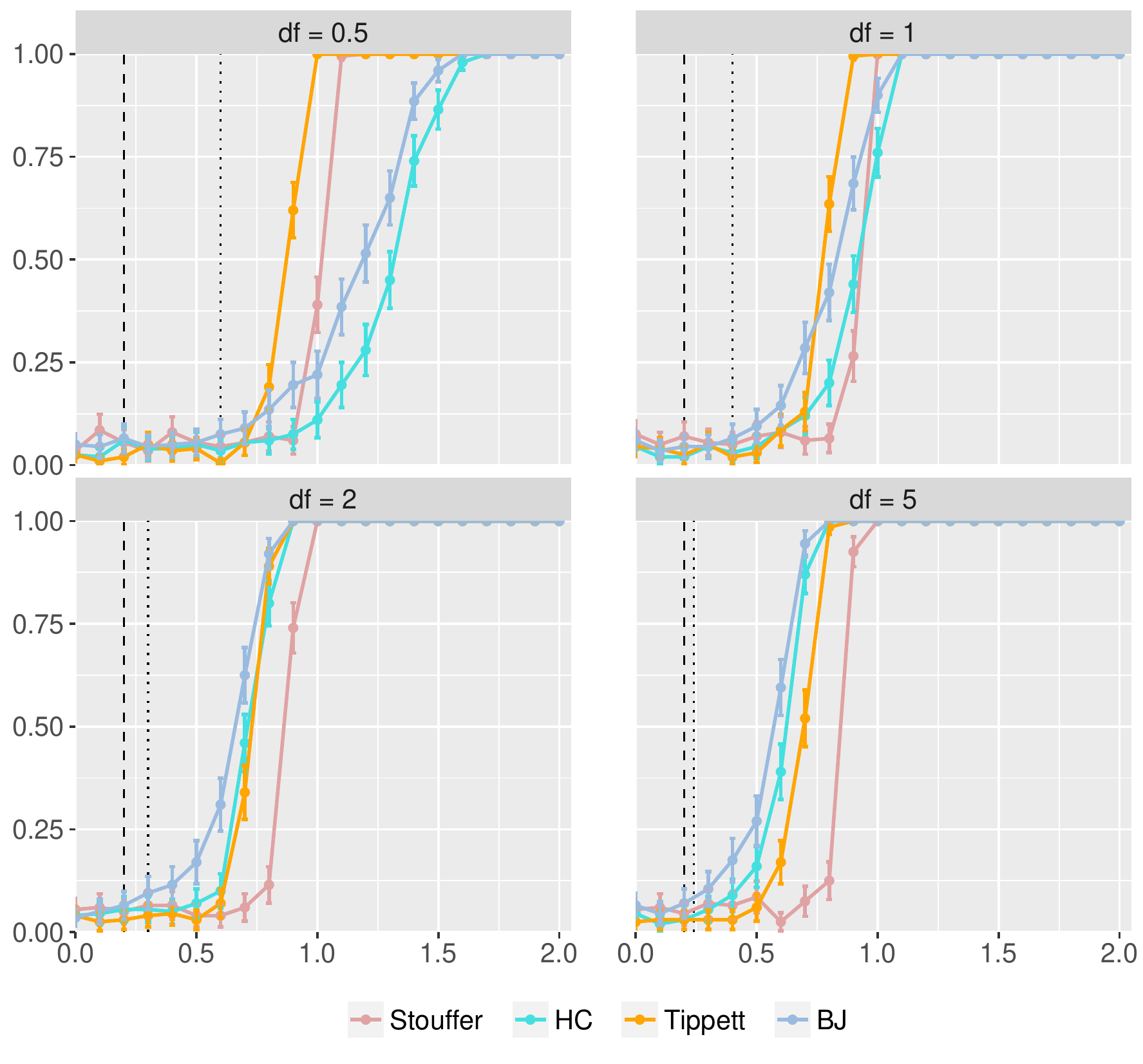}
\caption{Here $\beta = 0.6$, the x-axis represents $r$ in the parameterization \eqref{mu2}, y-axis the power of the tests identified in the legend.  Each subfigure corresponds to a Student t-distribution with the specified number of degrees of freedom.  The black dashed vertical line corresponds to the oracle scan detection boundary established in \prpref{power_law}, while the dotted line corresponds to the oracle threshold detection boundary established in \prpref{oracle_threshold_power_law}.}
\label{fig:6}
\end{figure}

\begin{figure}[h!]
\centering
\includegraphics[scale=0.6]{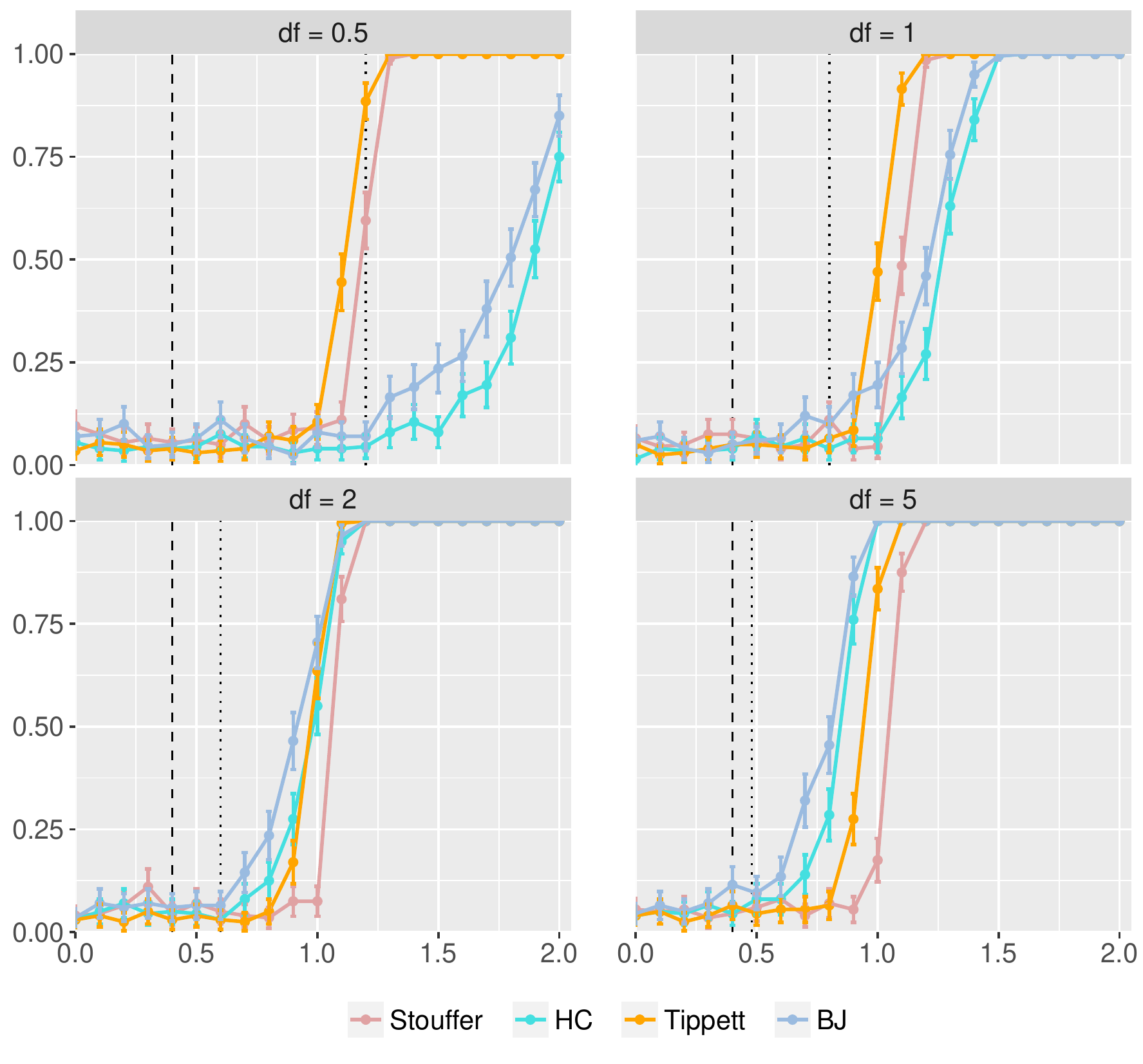}
\caption{Here $\beta = 0.7$, otherwise, see \figref{6} for more details.}
\label{fig:7}
\end{figure}

\begin{figure}[h!]
\centering
\includegraphics[scale=0.6]{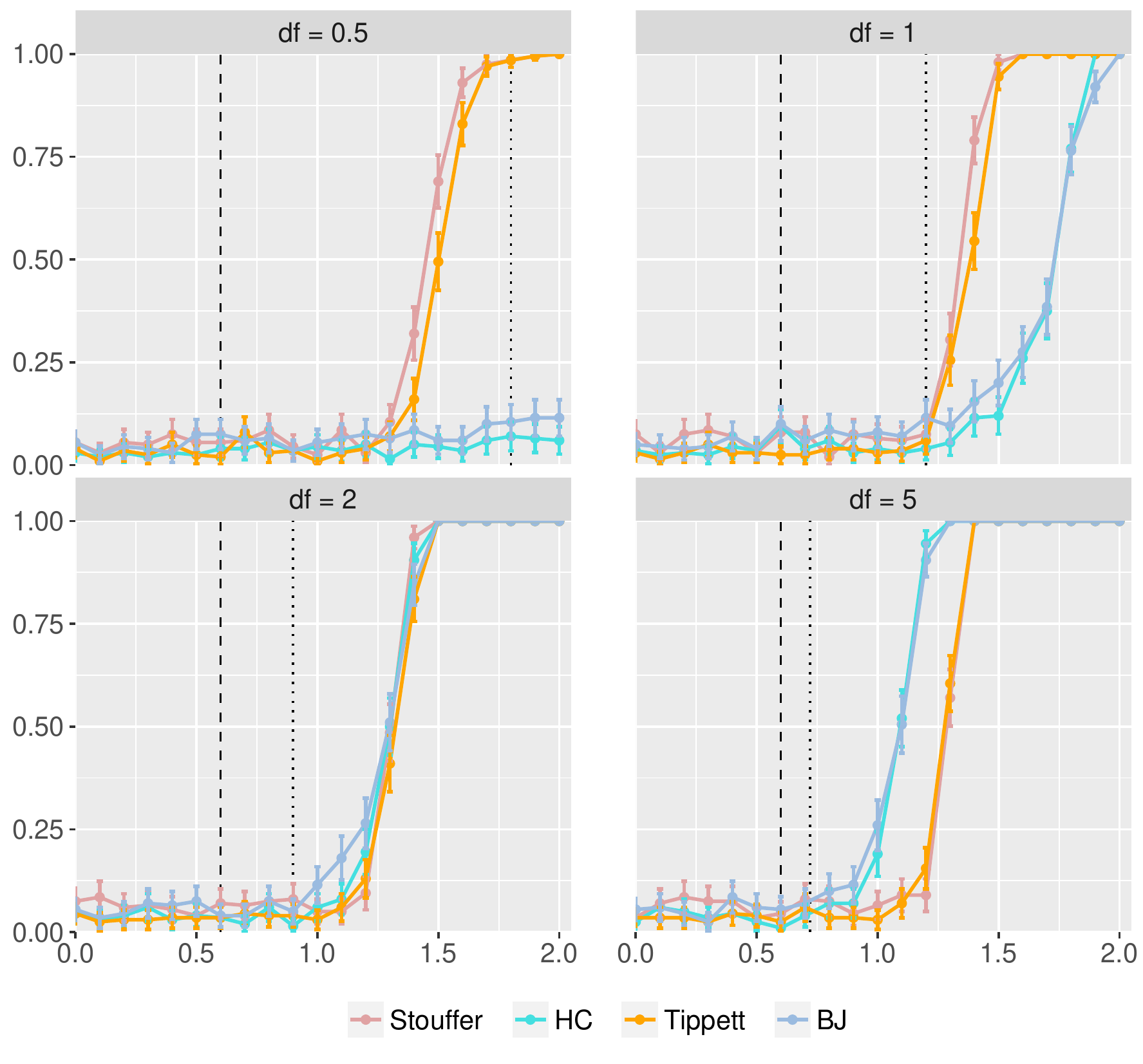}
\caption{Here $\beta = 0.8$, otherwise, see \figref{6} for more details.}
\label{fig:8}
\end{figure}

\bibliographystyle{abbrvnat}
\bibliography{ref}

\end{document}